\documentclass[a4paper,10pt]{article}
\usepackage{amsmath,amsthm,amssymb,amscd,epsfig,esint, mathrsfs,graphics,color}
 \usepackage{wrapfig}
\usepackage{verbatim}
\usepackage[english]{babel}
cm \oddsidemargin=.2true cm \topmargin=-.5truecm
\usepackage{amsfonts}
\usepackage{amssymb}
\usepackage{dsfont}
\usepackage{amsthm}
\usepackage{mathrsfs}
\usepackage{bbm}
\usepackage[english]{babel}
\usepackage{dsfont}
\usepackage{amssymb}
\usepackage{tipa}
\usepackage{textcomp}
\usepackage{pifont}

\usepackage[T1]{fontenc}
\usepackage{amsfonts}
\usepackage{indentfirst}
\usepackage{booktabs}
\usepackage{latexsym}
\usepackage{mathtools}
\usepackage{siunitx}
\usepackage{microtype}
\usepackage{braket}
\usepackage{hyperref}
\usepackage{subfig}
\usepackage{stmaryrd}
\usepackage{wrapfig}
\usepackage{setspace}
\usepackage{mathrsfs}
\usepackage[normalem]{ulem}
\usepackage{upgreek}
\usepackage{array}
\usepackage{enumerate}
\usepackage{tikz}
\usepackage{amstext}
\usepackage{latexsym}
\usepackage{blindtext}
\usepackage{colortbl}
\usepackage{color}
\usepackage{pgfplots}
\usepgfplotslibrary{fillbetween}

\pgfplotsset{width=10cm,compat=1.15}

\newcommand{\R}{\numberset{R}}

\theoremstyle{plain}
\newtheorem{thm}{Theorem}[section]

\theoremstyle{definition}

\def\XXint#1#2#3{{\setbox0=\hbox{$#1{#2#3}{\int}$}
        \vcenter{\hbox{$#2#3$}}\kern-.5\wd0}}

\def\dx{{\mathrm d}x}
\def\d\theta{{\mathrm d}\theta}

\def\R{\mathbb{R}}

\numberwithin{equation}{section} \makeatletter

\renewcommand{\p@enumi}{\thesection.}

\title{\sc{Regularity results for solutions to a class of obstacle problems}
\footnotetext{\hspace{-0.35cm} 2010 \emph{Mathematics Subject
Classification}. 35J87, 49J40.
\endgraf
{\it Key words and phrases}.  }}

\title{\textbf{
Energy approximation
for some double phase functionals}}

\author{
Menita Carozza\\
{\it  Department of Engineering, University of  Sannio}
\\ {\it Corso Garibaldi 107, 82100 Benevento, Italy}
\\ {\it e-mail: carozza@unisannio.it }
\bigskip
\\
 Filomena De Filippis\\
 {\it Department of  Mathematics  Physics and Informatics, University of Parma,  
}\\ {\it  Parco Area delle Scienze 53/a, 43124 Parma, Italy} \\{\it  e-mail: filomena.defilippis@unipr.it }
 \bigskip
 \\
  Raffaella Giova\\
 {\it Department of Economic and Legal Studies, 
University Parthenope of Napoli}\\ {\it Via Generale Parisi 13, 80132 Napoli,  Italy} 
\\{\it  e-mail: raffaella.giova@uniparthenope.it}
 \bigskip
\\
 Francesco Leonetti\\
 {\it Department of Information Engineering, Computer Science and Mathematics,}
 \\
 {\it University of l'Aquila}\\ {\it Via Vetoio snc, 
67100 L'Aquila, Italy} \\{\it  e-mail: francesco.leonetti@univaq.it}}
\begin{document}
	\maketitle
    \begin{abstract}
       \noindent
In the present manuscript we give a structure condition that allows for energy approximation in the vectorial case and we provide two examples, enjoying such a condition.

    \end{abstract}
	\section{Introduction}
	\noindent
Let us consider nonautonomous functionals as follows
\begin{equation}\label{functional}
  \mathcal{F}(u; \Omega):=\int_{\Omega}f(x, Du(x))\, \dx  
\end{equation}
where $\Omega$ is a bounded open subset of $\R^n$, $u:\Omega \rightarrow \R^N$, $f: \Omega \times \R^{N\times n} \rightarrow \R$,  $n\ge 2$ and $N \ge 1$. We assume that 
$x \to f(x,z)$ is measurable and $z \to f(x,z)$ is continuous; moreover, 
there exist positive constants $c_1,c_2, p, q$ with $1<p<q$ such that 
\begin{equation}\label{crescita pq}
  |z|^p \le f(x,z) \le c_1 |z|^q+c_2,
\end{equation} 
see (1.3) in \cite{Marcellini1989} and section 5 in \cite{Mingione_survey2006}.
In view of such $p,q$ growth, we assume that $u \in W^{1,p}_{loc}(\Omega;\mathbb{R}^N)$, with $x \to f(x,Du(x)) \in L^1_{loc}(\Omega)$.
Our model density $f$ is 
\begin{equation}\label{Zhikov}
    f(x, z)=|z|^p +a(x)|z|^q
\end{equation}
where $0\le a(x)\le c_3$, for some positive constant $c_3$.
When $a(x) = 0$ then  $|z|^{p} + a(x) |z|^{q} =  |z|^{p}$ and we are in the $\text{${p}$-phase}$. On the other hand, if $a(x) > 0$ then, for large $|z|$, we have that $|z|^{p} + a(x) |z|^{q} \approx  |z|^{q}$, so, we are in the \text{${q}$-phase}. Functional \eqref{functional}  with \eqref{Zhikov} has been considered  
by \cite{Zhikov}, \cite{EspLeoMin-2004}, 
\cite{FonMalMin}.
\noindent
Since then, a lot of people have been studying such a functional that is now known as double phase model, see \cite{Colombo-Mingione}, \cite{Colombo-Mingione-2}, \cite{Baroni-Colombo-Mingione}
\cite{Eleuteri-Marcellini-Mascolo}, 
\cite{Balci-Diening-Surnachev}, \cite{BCDFM}, \cite{Kinnunen-Nastasi-Camacho}.

\noindent
As we can see in \cite{EspLeoMin-2004}, \cite{DeFilippis-Mingione}, \cite{Koch}, 
\cite{DeFilippis-Mingione-Arch2023}, \cite{DeFilippis-Piccinini}, \cite{DDP}, an important step when proving regularity of minimizers $u$ for \eqref{functional} under \eqref{crescita pq} is energy approximation: if the ball $B$ is strictly contained in $\Omega$, there exists $u_k$ more regular than $u$ such that $u_k$ converges to $u$ and $\mathcal{F}(u_k;B) \rightarrow \mathcal{F}(u;B)$. 
It is well known that standard mollification gives us the strong convergence to $u$ in $W^{1,p}$; the difficult task is the energy approximation  $\mathcal{F}(u_k;B) \rightarrow \mathcal{F}(u;B)$ when $p<q$  
and dependence on $x$ appears in the density $f(x,z)$. 
On the other hand, absence of 
Lavrentiev phenomenon implies energy approximation for minimizers, see section \ref{en_approx_sec}. 
If $q$ is close to $p$ we get energy approximation for the double phase model \eqref{Zhikov}. {
Indeed, following \cite{Zhikov}, in \cite{EspLeoMin-2004} the authors prove} 
energy approximation 
for \eqref{Zhikov} when 
\begin{equation}\label{sigma holder}
    a \in C^{0,\sigma}, \quad 0<\sigma\le 1, \qquad q \le p\frac{n+ \sigma}{n}.
\end{equation} 
\\
Then, there have been some efforts in proving energy approximation for \eqref{functional} with a general density $f$ with $p,q$ growth \eqref{crescita pq}. Let us mention \cite{Esposito-Leonetti-Petricca} where $f(x,z)$ satisfies $p,q$ growth \eqref{crescita pq} and 
\begin{equation}\label{Es-Leo-Petr holder}
   |f(x,z)-f(y,z)|\le c_4|x-y|^\sigma(1+|z|^q),
\end{equation} 
where
\begin{equation}\label{Es-Leo-Petr sigma}
0<\sigma\le 1, \qquad q \le p\frac{n+ \sigma}{n};
\end{equation}
moreover, 

\begin{equation}\label{Es-Leo-Petr convex}
 z \rightarrow f(x,z), \,\, {\rm is \, convex}  
\end{equation}
and for every $x$  and  for every $\varepsilon>0$,  with $\overline{B(x,\varepsilon)}\subset \Omega$, there exists $y^* \in \overline{B(x,\varepsilon)}$ such that
  \begin{equation}\label{Es-Leo-Petr min ind z}   f(y^*,z) \le f(y,z), \qquad \forall y \in \overline{B(x,\varepsilon)},\quad  \forall z \in \R^{N\times n} .  
  \end{equation}
Let us remark that in general the minimum point $y^* \in \overline{B(x,\varepsilon)}$ of $y \rightarrow f(y,z)$ depends also on $z$. The assumption \eqref{Es-Leo-Petr min ind z} requires that $y^*$ does not depend on $z$; see also \cite{Koch}.
\noindent
Under assumptions \eqref{crescita pq}, \eqref{Es-Leo-Petr holder}, \eqref{Es-Leo-Petr sigma}, \eqref{Es-Leo-Petr convex} and \eqref{Es-Leo-Petr min ind z} they have energy approximation using mollification. \\
Functional  \eqref{functional} with $f(x,z)$ weak N-function is studied in \cite{BC}. A further step is obtained in \cite{BCM} where $f(x,z)$ is controlled by a weak N-function, more precisely 
\begin{equation}\label{hp structure BCM}
 \nu M(x,\beta z)\le f(x,z) \le L(M(x,z)+ g(x))
\end{equation}
for suitable constants $0<\nu, \beta <1 <L$ and functions $g(x), M(x,z)$ where $M$ is a weak N-function. 
\vspace{0,2cm}
Let us come back to the double phase model \eqref{Zhikov}; we already said that energy approximation is true when $a$ is $\sigma$-holder continuous and $\sigma \le 1$, see \eqref{sigma holder}. An important step has been done in \cite{BCDFM} where continuity of $a$ is no longer assumed. 
In \cite{BCDFM} the authors require that there exist 
constants $c_5\ge 0, \,  c_6\ge 1, \, \sigma>0$, such that 

\begin{equation}\label{weight a new}
0 \le a(x) \le c_6 a(\tilde{x})+c_5|x-\tilde{x}|^{\sigma},
\end{equation}
for every $x, \tilde{x}$. They denote the set of functions $a(x)$ satisfying \eqref{weight a new}  by $\mathcal{Z^\sigma}.$  If 
\begin{equation}\label{qp BCDFM}
    q\le p\frac{n+\sigma}{n},
\end{equation}
they prove that  Lavrentiev phenomenon does not occur.
Let us remark that \eqref{weight a new} requires the continuity of $a$ only at points $\tilde{x}$ where $a(\tilde{x})=0$; no continuity is required at points $\tilde{x}$ where $a(\tilde{x})>0$. Moreover, 
$\sigma$ is no longer bounded by 1. Note that, if $a \in C^{0,\sigma}$ then the right hand side of \eqref{weight a new} holds true with $c_6=1$ and $c_5= [a]_{0,\sigma}$ the $\sigma$-seminorm of $a$.
\noindent
In \cite{BCDFM} the authors deal also with more general densities $f(x,z)$ such that there exist positive constants $\nu_1,\nu_2 >0$, exponents $\tilde{p},\tilde{q}$ with $1\le \tilde{p} \le \tilde{q}$ and a function $\tilde{a} \in \mathcal{Z}^\sigma$ such that
\begin{equation}\label{hp_structure_BCDFM_new}
 \nu_1(|z|^{\tilde{p}}+ \tilde{a}(x)|z|^{\tilde{q}})\le f(x,z) \le \nu_2(|z|^{\tilde{p}}+ \tilde{a}(x)|z|^{\tilde{q}}).
\end{equation}
These nice results do not cover the next example. Indeed, let us consider 
\begin{equation}\label{a introduction}
a(x):=\begin{cases}
0 &\quad{\rm if}\quad   x_1 \leq 0   \\
(x_1)^{\sigma} &\quad{\rm if}\quad   0 < x_1 \leq r   \\
(x_1)^{\sigma} + h &\quad{\rm if}\quad   r < x_1  
\end{cases}
\end{equation}
where $0 < r$, $0 < \sigma$, $0 < h$ and $x=(x_1,...,x_n)$. Such a function $a$ satisfies \eqref{weight a new} with $c_5=c_6=(1+ \frac h {r^\sigma}) 2^\sigma$, see section \ref{examples}. Note that $a$ is not continuous at points $x=(x_1,...,x_n)$ with $x_1 = r$, so, we cannot use \cite{EspLeoMin-2004}, since \eqref{sigma holder} fails. Moreover, let us set  
\begin{equation}\label{es 1 introduzione}
 f(x,z)=|z|^p+a(x)(\max\{z_n^1; 0\})^q   
\end{equation}
with $1< p< q$. Then we cannot use \cite{Esposito-Leonetti-Petricca} since $a$ is not continuous at points $x=(x_1,...,x_n)$ with $x_1 = r$  and \eqref{Es-Leo-Petr holder} fails. Let us also note that we cannot apply \cite{BC} since $f(x,z)$ is not symmetric with respect $z$ so $f$ is not a weak N-function. Moreover, we cannot use \cite{BCM}, since \eqref{hp structure BCM} fails, see section \ref{examples}. Eventually, we cannot apply \cite{BCDFM}
since \eqref{hp_structure_BCDFM_new} fails, see section \ref{examples}.\\
The recent contribution \cite{BMT2024} assumes the following property, called (H$^{\alpha,L}$): given $\alpha \geq 1$ and $L>0$, there exist $A>0$, $b(x) \geq 0$, $\epsilon^* \in (0,1)$, $0 < \theta(\epsilon) \leq 1$ with $\lim\limits_{\epsilon \to 0} \theta(\epsilon) = 1$, such that, for every $\epsilon \in (0, \epsilon^*)$ and for all $x$, $z$, we have
\begin{equation}
    \label{H-alpha-L}
  |z|^\alpha +  (f^-_{x,\epsilon})^{**}(z) 
  \leq L \epsilon^{-n} 
  \Longrightarrow
  f(x,\theta(\epsilon) z) \leq A [(f^-_{x,\epsilon})^{**}(z) + b(x) +  |z|^p]
\end{equation}
where $f^-_{x,\epsilon}(z)$ is the essential infimum of $y \to f(y,z)$ when $y \in B(x,\epsilon)$; moreover,  $h^{**}$ is the highest convex function that does not exceed $h$. In Theorem 1 of \cite{BMT2024}, property \eqref{H-alpha-L} is assumed with $\alpha = p$. 
In Theorem 3 of \cite{BBCLM}, property \eqref{H-alpha-L} is called (H$^{\text{conv}})$ and assumed with $\alpha = \max\{p , n \}$, $b(x)=1$, $\theta(\epsilon)=1$, $\epsilon^* = \text{diam } \Omega$ and no $|z|^p$ on the right hand side. Does our example \eqref{es 1 introduzione} satisfies property \eqref{H-alpha-L}? More generally, how can we check the validity of \eqref{H-alpha-L}? Inspired by 
the behaviour of \eqref{Zhikov} when $a$ satisfies \eqref{weight a new}, 
let us 
consider densities {
$f$} such that

\begin{equation}\label{struttura nostra}
  |z|^p\le  f(x,z) \le K_1 f(\tilde{x},z)+K_2|x-\tilde{x}|^{\sigma}|z|^q+K_3 
\end{equation} 
with $K_1\ge 1, \,  K_2, K_3\ge 0,\,  \sigma>0$,  $1<p<q$ and $x, \tilde{x} \in \Omega, \; z \in \R^{N\times n}$. 
Note that if  we consider the model case \eqref{Zhikov} with $a(x)$ satisfying  \eqref{weight a new} then \eqref{struttura nostra} holds true with $K_1=c_6$, $K_2=c_5$ and $K_3=0$. Let us come back to 
\eqref{es 1 introduzione} with $a$ in \eqref{a introduction}, then  \eqref{struttura nostra}  holds true with $K_1=K_2=(1+ \frac h {r^\sigma}) 2^\sigma$ and $K_3=0$. Note that density \eqref{es 1 introduzione} is inspired by \cite{TangQi}, where the 
author is facing the behaviour of some reinforced materials.
Now, in order to check the validity of \eqref{H-alpha-L},  we establish the following Theorem.

\begin{thm}\label{first_result}

Let us consider a bounded open subset  $\Omega$ of 
 $\R^n$, $f: \Omega \times \R^{N\times n} \rightarrow \R$,  $n\ge 2$ and $N \ge 1$.
Assume that there exist constants $K_1\ge 1, \,  K_2, K_3\ge 0,\,  \sigma>0$,   $1<p<q$ 
such that 

$$ 
q\le p \left (1+\frac{\sigma}{n} \right )
\eqno{\rm (F1)}
$$
and
$$ |z|^p \le f(x,z) \le K_1 f(\tilde{x},z)+K_2|x-\tilde{x}|^{\sigma}|z|^q+K_3 \eqno{\rm (F2)}$$
for every $x, \tilde{x} \in \Omega$, $z \in \R^{N\times n} $. Moreover, we require that, for every $x\in \Omega$,
$$  z \rightarrow f(x,z) \,\, {\rm is \, convex .} \eqno{\rm (F3)}$$


\vspace{0.2cm}
\noindent Furthermore, for every $x\in \Omega$  and  for every $\varepsilon>0$,  with $\overline{B(x,\varepsilon)}\subset \Omega$, there exists $y^* \in \overline{B(x,\varepsilon)}$ such that 
  $$   f(y^*,z) \le f(y,z), \qquad \forall y \in \overline{B(x,\varepsilon)},\quad  \forall z \in \R^{N\times n} . \eqno{\rm (F4)}$$
Then, property \eqref{H-alpha-L} holds true with $\alpha = p$, $\theta(\epsilon) = 1$, $A=K_1 + K_2 L^{\frac{q-p}{p}}$, $b(x) = K_3$ and any $\epsilon^* \in (0,1)$, provided $\overline{B(x,\epsilon)} \subset \Omega$.
\end{thm}
\noindent
The proof of this theorem will be made in 
section \ref{sec_proof_first_result}.

\noindent
We have to say that \cite{BMT2024} and \cite{BBCLM}  allow $f$ to depend on $u$ as well: $f=f(x,u,z)$. 
Moreover, both \cite{BMT2024} and \cite{BBCLM} want that the energy approximating functions $u_k$ keep fixed the boundary values of $u$, so a lot of extra work has to be done. The proofs are done in the scalar case $N=1$. Indeed, in Theorem 1 of \cite{BMT2024}, the authors use the minimum of a finite number of functions: see pages 2283 and 2301 of \cite{BMT2024};  this approach
 does not seem to extend to the vectorial setting $N \geq 2$. In Theorem 3 of \cite{BBCLM}, the authors first approximate, in energy, the starting function $u$ by means of bounded functions $u_k$, see page 43 in \cite{BBCLM}. This is accomplished by a truncation method: $u_k = u$ where $|u| \leq k$, $u_k = k$ where $u>k$ and $u_k = -k$ where $u<-k$. So, $Du_k = 0$ on $|u|>k$ and 
 $$
 \int\limits_{ \Omega} f(x, Du_k(x)) dx
 =
\int\limits_{\{ |u| \leq k\} }
f(x, Du(x)) dx
+
\int\limits_{\{ |u| > k\}}
f(x, 0) dx.
$$
If we know that $x \to f(x,0) \in L^1(\Omega)$, then $\int\limits_{\{ |u| > k\}}
f(x, 0) dx$ goes to $0$ and we have energy approximation.
This works in the scalar case $N=1$. This approach does not seem to extend to the vectorial setting $N \geq 2$, unless we assume Uhlenbeck structure $f=\tilde{f}(x,|z|)$. Indeed, the vectorial truncation on the superlevel set $\{|u|>k\}$ is given by $u_k(x) = k \frac{u(x)}{|u(x)|}$, so the gradient is no longer zero but we have 
$$
D_i u_k^\alpha(x) = \frac{k}{|u(x)|} 
\left[
D_i u^\alpha(x) -  
\frac{u^\alpha(x)}{|u(x)|} 
\sum\limits_{\beta=1}^{N}
\frac{u^\beta(x)}{|u(x)|}
D_i u^\beta(x)
\right].
$$
So, we are left to manage the integral
$$
\int\limits_{\{ |u| > k\}}
f(x, Du_k(x)) dx
=
\int\limits_{\{ |u| > k\}}
f
\left(
x,\frac{k}{|u(x)|} 
[ 
D_i u^\alpha(x) -  
\frac{u^\alpha(x)}{|u(x)|} 
\sum\limits_{\beta=1}^{N}
\frac{u^\beta(x)}{|u(x)|}
D_i u^\beta(x)
)
\right) dx. 
$$
If we have Uhlenbeck structure $f=\tilde{f}(x,|z|)$ with $s \to \tilde{f}(x,s)$ increasing, since
$$
\left|
Du_k(x)  
\right|
=
\left|
\frac{k}{|u(x)|} 
\left[
D_i u^\alpha(x) -  
\frac{u^\alpha(x)}{|u(x)|} 
\sum\limits_{\beta=1}^{N}
\frac{u^\beta(x)}{|u(x)|}
D_i u^\beta(x)
\right]
\right|
\leq
\left|
D u(x)  
\right|,
$$
we have
$$
f(x,Du_k(x))=\tilde{f}(x,|Du_k(x)|) \leq \tilde{f}(x,|Du(x)|) = f(x,Du(x)).
$$
Since $x \to f(x,Du(x)) \in L^1(\Omega)$, we have that
$\int\limits_{\{ |u| > k \} } f(x,Du(x)) dx$ goes to $0$ and we can argue as in the scalar case, see page 234 in \cite{BuGwSk}.

\noindent
Let us come back to 
Theorem 1 of \cite{BMT2024} and Theorem 3 of \cite{BBCLM}: they are in the scalar case $N=1$. 
When proving regularity of minimizers, energy approximation is an important tool but we do
not need the same boundary value for all the approximating functions $u_k$. This makes things easier. 
On the other hand, dealing with the vectorial case $N \geq 1$ is also important. So, for future reference, we explicitely write 
 and prove a result in such a framework.

\begin{thm}\label{main result}

Let us consider the functional 
\[
\mathcal{F}(u; \Omega):=\int_{\Omega}f(x, Du(x))\, \dx
	\]
where $\Omega$  is a bounded open subset of 
 $\R^n$, $u:\Omega \rightarrow \R^N$, $f: \Omega \times \R^{N\times n} \rightarrow \R$,  $n\ge 2$ and $N \ge 1$. Moreover, $x \to f(x,z)$ is measurable. 
Assume (F1), (F2), (F3) and (F4). 
Let $u$ belong to $W^{1,p}_{loc}(\Omega; \R^N)$ with $x \to f(x,Du(x)) \in L^1_{loc}(\Omega)$ and let $u_\varepsilon$ be the standard mollification of $u$. Then, 
for every open ball $B=B(x^*,\rho)$ with center at $x^*$, radius $\rho$, with $\overline{B} \subset \Omega$, we have 
$u_\varepsilon \rightarrow u$ in $ W^{1,p}(B; \R^N)$ and 
$$\mathcal{F}(u_\varepsilon; B) \rightarrow \mathcal{F}(u; B).
\eqno{\rm (F5)}$$
\end{thm}
\noindent
The proof of this theorem will be made in section \ref{sec proof}.

\noindent
We remark that, 
 if $a(x)$ is sequentially lower 
 semicontinuous in $\Omega$
and
$$
f(x,z)=f_1(z)+ a(x)f_2(z)
$$ 
with $f_2(z)\ge 0$, then, property (F4) holds true. 
We ask ourselves whether there are $f(x,z)$ without product structure $ a(x)f_2(z)$,  satisfying property {(F4)}. The answer is positive. Indeed, for $q>1$, let us consider
\begin{equation}
\label{g_no_product}
g(x_1,t)=\max\{(\max \{t,0\})^q - (\max \{x_1,0\})^q;\, 0\}
\end{equation}
and, for $1<p<q$, 
\begin{equation}
\label{f_no_product}
f(x,z)=|z|^p+ g(x_1,z_n^1)=|z|^p+\max\{(\max \{z_n^1,0\})^q - (\max \{x_1,0\})^q;\, 0\}.
\end{equation}
Then $g(x_1,t)$ has no product structure 
$a(x_1)h(t)$, see section \ref{examples}. On the other hand,
all the assumptions of Theorem \ref{main result} are fulfilled with $\Omega = B(0, 1)$, $\sigma = n(q-p)/p$, $K_1 = 1$, $K_2 = 0$, $K_3 = 1$. In particular, $q$ can be far from $p$ as much as we like, since $\sigma$ can be chosen to be $n(q-p)/p$ and (F1), (F2) are satisfied; note that $K_2$ is $0$.

\noindent
In section \ref{examples} we will do calculations for \eqref{es 1 introduzione} and for \eqref{f_no_product}.

\section{Proof of Theorem \ref{first_result}}\label{sec_proof_first_result}
\noindent
Let us fix $x$ and $\epsilon$; assumption (F4) guarantees the existence of $y^*$ such that $|x - y^*| \leq \epsilon$ and $f(y^*,z) \leq f(y,z)$ for all $y \in \overline{B(x,\epsilon)}$ and all $z$. Then 
$f(y^*,z) \leq \text{essential infimum}\{ f(y,z): y \in B(x,\epsilon)\}  = f^-_{x,\epsilon}(z)$. Convexity (F3) of $z \to f(y^*,z)$ says that 
\begin{equation}
  \label{controlling_f(y^*,z)}  
0 \leq f(y^*,z) \leq (f^-_{x,\epsilon})^{**}(z).
\end{equation}
Now we use (F2) with $\tilde{x} = y^*$ and we get
\begin{equation}
    \label{controlling_f(x,z)}
 f(x,z) \leq K_1 f(y^*,z)+K_2|x-y^*|^{\sigma}|z|^q+K_3 
 \leq
K_1 (f^-_{x,\epsilon})^{**}(z) + K_2 \epsilon^{\sigma} |z|^q + K_3, 
\end{equation}
where we used \eqref{controlling_f(y^*,z)} and $|x-y^*| \leq \epsilon$. Note that the right hand side of 
\eqref{H-alpha-L} has to be verified only when 
$|z|^\alpha +  (f^-_{x,\epsilon})^{**}(z) 
  \leq L \epsilon^{-n}$. Since $0 \leq (f^-_{x,\epsilon})^{**}(z)$, we have 
$|z|^\alpha \leq L \epsilon^{-n}$. If we take $\alpha = p$ we get
$$
|z|^p \leq L \epsilon^{-n}.
$$
Now we can write 
$$
\epsilon^\sigma |z|^q = \epsilon^\sigma |z|^{p} |z|^{q-p} \leq \epsilon^\sigma |z|^{p} ((L \epsilon^{-n})^\frac{1}{p})^{q-p} = L^{\frac{q-p}{p}} \epsilon^{\sigma - \frac{(q-p)n}{p}} |z|^p.
$$
Assumption (F1) is equivalent to $\sigma - \frac{(q-p)n}{p} \geq 0$. Fix any $\epsilon^* \in (0,1)$, then, for every $\epsilon \in (0, \epsilon^*)$, we have 
$$
\epsilon^{\sigma - \frac{(q-p)n}{p}} \leq 
(\epsilon^*)^{\sigma - \frac{(q-p)n}{p}}
\leq 1.
$$
We have shown that
$$
\epsilon^\sigma |z|^q \leq L^{\frac{q-p}{p}} |z|^p.
$$
We insert this estimate into \eqref{controlling_f(x,z)} and we get
\begin{equation}
    \label{controlling_f(x,z)_bis}
 f(x,z) 
 \leq
K_1 (f^-_{x,\epsilon})^{**}(z) + K_2 L^{\frac{q-p}{p}} |z|^p + K_3. 
\end{equation}
This means that \eqref{H-alpha-L} holds true with $\alpha = p$, $\theta(\epsilon) = 1$, $A=K_1 + K_2 L^{\frac{q-p}{p}}$, $b(x) = K_3$ and any $\epsilon^* \in (0,1)$, provided $\overline{B(x,\epsilon)} \subset \Omega$.
\qed

\section{Proof of Theorem \ref{main result}}\label{sec proof}

\noindent
Since 
$\overline{B}=\overline{B(x^*,\rho)} \subset \Omega$, then there exists $R>\rho$ such that $\overline{B(x^*,R)} \subset \Omega$. Let us write $B_r$ in place of $B(x^*,r)$. 
Let $x \in B_{\rho}$ and 
$\tilde{x} \in \overline{B(x,\varepsilon)}$ with $0<\varepsilon< R-\rho$. We  consider the standard mollification kernel $\phi$. With $\phi_{\varepsilon}(t)=\phi(\frac t \varepsilon) \frac{1}{\varepsilon^n}$,  we recall

\[
u_\varepsilon(x)= \int_{B(x,\varepsilon)}u(y)\phi_{\varepsilon}(x-y) dy,
\]

\[
Du_\varepsilon(x)= \int_{B(x,\varepsilon)}Du(y)\phi_{\varepsilon}(x-y) dy,
\]
and we  estimate, using H\"older inequality, $\overline{B(x,\varepsilon)}\subset B_{\rho + \varepsilon} \subset B_{R}$ and the change of variable $t=\frac{1}{\varepsilon}(x-y)$, as follows
 
\begin{align}\label{gradmol}
|Du_\varepsilon(x)|= \biggm|\int_{B(x,\varepsilon)}Du(y)\phi_{\varepsilon}(x-y) dy\biggm|\le& \biggm(\int_{B(x,\varepsilon)} |Du(y)|^p dy\biggm)^{\frac 1 p} \biggm(\int_{B(x,\varepsilon)} [\phi_{\varepsilon}(x-y)]^{p'} dy\biggm)^{\frac 1 {p'}}\notag \\
\le& \biggm(\int_{B_{\rho + \varepsilon}} |Du(y)|^p dy\biggm)^{\frac 1 p} \biggm(\int_{B(0,1)} [\phi(t)]^{p'} dt\biggm)^{\frac 1 {p'}} \cdot \, \varepsilon^{- \frac{n}{p}}\notag \\
\le& \biggm(\int_{B_{R}} |Du(y)|^p dy\biggm)^{\frac 1 p} \biggm(\int_{B(0,1)} [\phi(t)]^{p'} dt\biggm)^{\frac 1 {p'}} \cdot \,\varepsilon^{- \frac{n}{p}}.\notag \\
\end{align}

\noindent By assumption (F2),  it holds

\[
0\le f(x,Du_\varepsilon(x))\le K_1 f(\tilde{x},Du_\varepsilon(x))+K_2|x-\tilde{x}|^{\sigma}|Du_\varepsilon(x)|^q +K_3,
\]

\noindent and, by virtue of \eqref{gradmol}, $|Du_\varepsilon(x)|\le c_1 \varepsilon^{- \frac{n}{p}}$, so that $K_2|x-\tilde{x}|^{\sigma}|Du_\varepsilon(x)|^q \le K_2 \varepsilon^{\sigma} \cdot ( c_1 \varepsilon^{- \frac{n}{p}})^{q-p}|Du_\varepsilon(x)|^p$.
Note that (F1) is equivalent to $\sigma - \frac{n}{p}(q-p)\ge 0 $, so that, if
$\varepsilon \le 1$, we have 
$\varepsilon^{\sigma - \frac{n}{p}(q-p)} \leq 1$
and
\[
0\le f(x,Du_\varepsilon(x))\le K_1 f(\tilde{x},Du_\varepsilon(x))+K_2\cdot c_1^{q-p}|Du_\varepsilon(x)|^p +K_3.
\]

\noindent Using the left hand side of (F2),  setting $c_2 = K_2\cdot c_1^{q-p}$ and $c_3=K_1 + c_2$, from the previous inequality we can deduce that 
$ \forall x \in B_{\rho}$, $\forall \tilde{x} \in \overline{B(x,\varepsilon)}$ and $ \forall \varepsilon \in (0, 1\wedge (R-\rho))$

\begin{equation}\label{inequality 1}
 0\le f(x,Du_\varepsilon(x))\le K_1 f(\tilde{x},Du_\varepsilon(x))+c_2 f(\tilde{x},Du_\varepsilon(x))+K_3= c_3 f(\tilde{x},Du_\varepsilon(x))+K_3.
\end{equation}

\noindent Let us choose  $\tilde{x}=y^*\in \overline{B(x,\varepsilon)}$ given by assumption (F4), and, for every fixed  $x, \varepsilon$, let us  consider the function
\[
H(z)=f(y^*,z).
\]
Then, we can rewrite \eqref{inequality 1} as follows

\begin{equation}\label{inequality 2}
 0\le f(x,Du_\varepsilon(x))\le  c_3 f(\tilde{x},Du_\varepsilon(x))+K_3=c_3 f(y^*,Du_\varepsilon(x))+K_3=c_3 H(Du_\varepsilon(x))+K_3.
\end{equation}
Assumption (F3) says that $H$ is convex, so, by Jensen inequality,

\begin{align}\label{H}
H(Du_\varepsilon(x))= & H \biggm( \int_{B(x,\varepsilon)}Du(y)\phi_{\varepsilon}(x-y) dy \biggm)\le \int_{B(x,\varepsilon)}H(Du(y))\phi_{\varepsilon}(x-y) dy \notag \\
=& \int_{B(x,\varepsilon)}f(y^*, Du(y))\phi_{\varepsilon}(x-y) dy \notag \\
\le & \int_{B(x,\varepsilon)}f(y, Du(y))\phi_{\varepsilon}(x-y) dy \notag \\
=& \int_{B(x,\varepsilon)}h(y)\phi_{\varepsilon}(x-y) dy=h_{\varepsilon}(x)
\end{align}
where we have also used the minimality of $y^*$ from (F4), and 
we have set $h(y)=f(y,Du(y))$. Note that $h \in L^1(B_{R})$.
From  \eqref{inequality 2} and \eqref{H} we have,  for every $x \in B_{\rho}$ and for every $  \varepsilon \in (0, 1\wedge  (R-\rho))$,
\[
0\le f(x,Du_\varepsilon(x)) \le c_3 h_{\varepsilon}(x)+K_3.
\]
Moreover,
\[
Du_\varepsilon\rightarrow Du \,\, {\rm strongly\, in} \,\, L^p(B_{\rho}),
\]

\[
h_\varepsilon \rightarrow h \,\, {\rm strongly\, in} \,\,L^1(B_{\rho}),
\]

\[
\int_{B_{\rho}} h_\varepsilon dx \rightarrow \int_{B_{\rho}} h dx.
\]
\noindent
Note that mollification gives pointwise convergence at Lebesgue points, so
\[
Du_\varepsilon(x)\rightarrow Du(x) \,\quad {\rm and} \quad\, h_\varepsilon(x)\rightarrow h(x)\quad {\rm  a.e.\quad in}\quad  B_{\rho}.
\]
\noindent
Convexity assumption (F3) implies the continuity of $z \rightarrow f(x,z)$ and, thanks to pointwise convergence a.e., we get
\[
f(x,Du_\varepsilon(x)) \rightarrow f(x,Du(x)),\quad {\rm  a.e.\quad in}\quad  B_{\rho}.
\]
The Dominated Convergence Theorem implies
\[
\int_{B_{\rho}} f(x,Du_\varepsilon (x))dx \rightarrow  \int_{B_{\rho}} f(x,Du(x))dx
\]
as claimed. This ends the proof of Theorem \ref{main result}.
\qed

\section{Examples}\label{examples}
\noindent
 Let us consider $\tilde{a}: \mathbb{R} \to [0, +\infty)$ as follows

\begin{equation*}
\tilde{a}(t):=\begin{cases}
0 &\quad{\rm if}\quad   t\leq r_1   \\
(t-r_1)^{\sigma} &\quad{\rm if}\quad r_1 < t \leq r_2 \\
(t-r_1)^{\sigma}+h &\quad {\rm if}\quad t>r_2 
\end{cases}
\end{equation*}
with  $r_1 < r_2$, $\sigma >0,\, h>0
$.

\noindent Let us verify that $t\to \tilde{a}(t)$  belongs to the class $\mathcal Z^\sigma$, i.e.
$$0 \le \tilde{a}(t) \le c_6 \tilde{a}(s)+c_5|t-s|^{\sigma},$$
for every $t,s \in \mathbb{R}$, for some constants $c_5 \geq 0, c_6 \geq 1$.

\noindent
{\bf Case 1}: $t \leq r_1$. Then,

$$
\tilde{a}(t)=0 \leq \tilde{a}(s),
$$
for every $s$.

\noindent
{\bf Case 2}: $ r_1 < t \leq r_2$.

\noindent
{\bf Subcase 2.1}: $s \leq r_1$. Then $\tilde{a}(s)=0$, so

$$
\tilde{a}(t) = (t-r_1)^\sigma =  \tilde{a}(s) + (t-r_1)^{\sigma} \leq \tilde{a}(s) + |t-s|^{\sigma}.
$$

\noindent
{\bf Subcase 2.2}: $s > r_1$. We observe that $\tilde{a}(s) \geq (s-r_1)^\sigma$. Then
$$
\tilde{a}(t) = (t-r_1)^{\sigma} \leq 2^\sigma [|t-s|^\sigma + (s-r_1)^{\sigma}] \leq 2^\sigma [\tilde{a}(s)+|t-s|^\sigma].
$$
\noindent
{\bf Case 3}: $ t > r_2$.  Then
\begin{align} \label{3.1}
    \tilde{a}(t) & = (t-r_1)^\sigma + h \notag \\ & = (t-r_1)^\sigma + \frac{h}{(r_2 - r_1)^\sigma}(r_2 - r_1)^\sigma \notag \\ & \leq (t-r_1)^\sigma + \frac{h}{(r_2 -r_1)^\sigma}(t-r_1)^\sigma \notag \\ & = \left [ 1+\frac{h}{(r_2 - r_1)^\sigma}\right ](t-r_1)^\sigma.
\end{align}

\noindent
{\bf Subcase 3.1}: $s \leq r_1$. Then $\tilde{a}(s)=0$, and so by \eqref{3.1} we have

$$
\tilde{a}(t) \leq \left [ 1+\frac{h}{(r_2 - r_1)^\sigma} \right ]|t-s|^\sigma = \tilde{a}(s) + \left [ 1+\frac{h}{(r_2 - r_1)^\sigma} \right ]|t-s|^\sigma. 
$$

\noindent
{\bf Subcase 3.2}: $s > r_1$. As in subcase 2.2, we observe that
$$
\tilde{a}(t) \leq  2^\sigma \left [ 1+\frac{h}{(r_2 - r_1)^\sigma} \right ] (|t-s|^\sigma + (s-r_1)^\sigma) \leq  2^\sigma \left [ 1+\frac{h}{(r_2 - r_1)^\sigma} \right ] (\tilde{a}(s)+|t-s|^\sigma ).
$$ 

\noindent
In all cases, we have that 

$$
0\le \tilde{a}(t)\le  2^\sigma \left [ 1+\frac{h}{(r_2 - r_1)^\sigma} \right ] (\tilde{a}(s)+|t-s|^\sigma ),
$$
for every $t,\, s$. Then \eqref{weight a new} is satisfied with $c_5=c_6= 2^\sigma \left [ 1+\frac{h}{(r_2 - r_1)^\sigma} \right ]$.

\noindent Remark that $\tilde{a}$ is sequentially lower semicontinuous in $\mathbb{R}$. Indeed, if $t_k\to t$, 

\noindent if $t \leq r_1$, 
 then 
 $$\tilde{a}(t)= 0 \leq \liminf_k \tilde{a}(t_k);$$

\noindent if $r_1 < t \leq r_2$,
 then, there exists $k_0$ such that, for every $k\ge k_0, \, r_1 <  t_k$ and
 $$\tilde{a}(t)=(t-r_1)^{\sigma}=\lim_k (t_k-r_1)^\sigma \leq \liminf_k \tilde{a}(t_k);$$

\noindent if $r_2 < t$, 
 then, there exists $k_0$ such that for every $k\ge k_0, \,  r_2 < t_k$ and 
 $$\tilde{a}(t)=(t-r_1)^\sigma+h=\lim_k (t_k-r_1)^\sigma+h=\lim_k \tilde{a}(t_k).$$

\noindent
Therefore, $\tilde{a}$ is  sequentially lower semicontinuous, but it is not continuous at $r_2$.

 \noindent
 Let us set

 $$
 a(x) = \tilde{a}(x_1),
 $$
 where $x=(x_1,...,x_n)$; 
 then $a$ is sequentially lower semicontinuous in $\mathbb{R}^n$ but it is not continuous at points $x$ with $x_1 = r_2$. Moreover, 

 $$
0\le a(x) = \tilde{a}(x_1)\le  2^\sigma \left [ 1+\frac{h}{(r_2 - r_1)^\sigma} \right ] (\tilde{a}(y_1)+|x_1 - y_1|^\sigma )
\leq 
2^\sigma \left [ 1+\frac{h}{(r_2 - r_1)^\sigma} \right ] ({a}(y)+|x-y|^\sigma ),
$$
for every $x,\, y \in \mathbb{R}^n$. Then \eqref{weight a new} is satisfied with $c_5=c_6= 2^\sigma \left [ 1+\frac{h}{(r_2 - r_1)^\sigma} \right ]$.

 \noindent  { \bf Remark 1.}
 Let us consider 
 $$
 f(y,\, z)=f_1(z)+a(y)f_2( z), \qquad\quad f_2\ge 0\,,
 $$  with $a$  sequentially lower semicontinuous. 
 Then, 
   $\forall x$ and $\forall\varepsilon$,  there exists $y^*=y^*(x,\,\varepsilon)\in \overline{B(x,\,\varepsilon)} $ such that 
 $$a(y^*)\le a(y)
  \qquad \forall y\in \overline{B(x,\,\varepsilon)}
  $$
  so that
  $$
f(y^*,\, z)\le f(y,\,z)
\qquad \forall y\in \overline{B(x,\,\varepsilon)}
\quad { \rm and } 
\quad \forall z
$$
 
\vspace{0.3 cm}

\noindent
{\bf Example 1}

\noindent
Let us consider
\[
f(x,z)=|z|^p+a(x)(\max\{z_n^1; 0\})^q
\]
with $1 < p< q$, and $0\le a(x) \le A $, $a \in \mathcal{Z}^\sigma$. Then, using \eqref{weight a new}, we get
\begin{align*}
   f(x,z)= &|z|^p+a(x)(\max\{z_n^1; 0\})^q  \\ \notag
  \le &|z|^p+ [c_6 a(\tilde{x})+ c_5 |x-\tilde{x}|^\sigma] (\max\{z_n^1; 0\})^q  \\ \notag
  \le & c_6|z|^p+ c_6 a(\tilde{x}) (\max\{z_n^1; 0\})^q + c_5 |x-\tilde{x}|^\sigma |z|^q  \\ \notag
  = & c_6 f(\tilde{x}, z)+  c_5 |x-\tilde{x}|^\sigma |z|^q,  \\ \notag
\end{align*} 
so that \eqref{struttura nostra} holds true with $K_1 = c_6$, $K_2 = c_5$, $K_3 = 0$.

\noindent The density does not have  Uhlenbeck structure.
Indeed, if we consider the matrix $z$ whose entries are

\[
z_n^1=1,   \qquad z_j^{\alpha}=0 \qquad {\rm otherwise,}
\]
then
\[
|z|=1, \qquad f(x,z)= 1+a(x)
\]
On the other hand, considering the matrix $\tilde{z}=-z$, we have

\[
\tilde{z}_n^1=-1,  \qquad \tilde{z}_j^{\alpha}=0   \qquad {\rm otherwise,}
\]
then
\[
|-z|=|z|=1, \qquad f(x,-z)= 1+a(x)(\max\{-1; 0\})^q=1.
\]
Supposing $f(x,\xi)=g(x, |\xi|)$ \, $\forall x, \xi$, we would have

\begin{align*}
\xi=& z \qquad \Rightarrow \qquad 1+a(x)=g(x,1)\notag\\
\xi= &\tilde{z}=-z \qquad \Rightarrow \qquad 1=g(x,1)
\end{align*}
and then

$$a(x)=0 \quad \forall x,$$
but we are interested in  double phase functionals, so $a(x)>0$ for some $x$.\\
\noindent The previous calculations show that $f(x,z)$ is not symmetric with respect to $z$, so we cannot use the result in \cite{BC}.

\noindent We cannot apply the result in \cite{BCDFM} to this functional since \eqref{hp_structure_BCDFM_new} does not hold true. Indeed, by contradiction, let us assume \eqref{hp_structure_BCDFM_new}. Then, if we consider the matrix $z$ such that

\[
z_1^1=t, \quad t>0; \qquad z_j^{\alpha}=0 \qquad {\rm otherwise}
\]
then $f(x,z)=|z|^p=t^p$ and, using the left hand side of \eqref{hp_structure_BCDFM_new}, we get
\[
\nu_1\bigg( t^{\tilde{p}}+ \tilde{a}(x)t^{\tilde{q}}\bigg) \le  t^p
\]
therefore as $t \rightarrow \infty$
\begin{align}\label{a(x)bis}
{\rm if }  \quad  \tilde{a}(x)=0&  \qquad {\rm then}  \qquad \tilde{p} \le p,\notag\\
{\rm if} \quad \tilde{a}(x)>0 & \qquad  {\rm then}  \qquad \tilde{q} \le p.
\end{align}
If we consider the matrix $\tilde{z}$ given by
\[
\tilde{z}_n^1=t, \quad t>0; \qquad \tilde{z}_j^{\alpha}=0 \qquad {\rm otherwise}
\]
then, $|\tilde{z}|=t$ and $f(x, \tilde{z})= t^p+ a(x)t^q$, provided 
$x$ has been selected in such a way that $a(x)>0$. Using the right hand side of \eqref{hp_structure_BCDFM_new}, we get
\[
t^p+ a(x)t^q \le \nu_2\bigg( t^{\tilde{p}}+ \tilde{a}(x)t^{\tilde{q}}\bigg).
\]
Therefore, as $t \rightarrow \infty$, if $\tilde{a}(x)=0$  then $ q \le \tilde{p}$ that is impossible since by \eqref{a(x)bis} we get $\tilde{p} \le p< q$. On the other hand, 
if $\tilde{a}(x)>0$ then $q \le \tilde{q}$ that is impossible since \eqref{a(x)bis} gives $\tilde{q} \le p< q$.

\vspace{0.2 cm}
\noindent We cannot also apply the result in \cite{BCM} to the  functional above. Indeed, in \cite{BCM} the authors assume that there exist constants $0<\nu, \beta <1 <L$ such that
\begin{equation}\label{hp structure BCM new}
 \nu M(x,\beta z)\le f(x,z) \le L(M(x,z)+ g(x))
\end{equation}
for suitable $g(x), M(x,z)$ where $M(x,-z)=M(x,z)$. We will show that \eqref{hp structure BCM new}
is not satisfied.
To this aim, we argue by contradiction: we assume  \eqref{hp structure BCM new}  and we consider the matrix $z$ such that
\[
z_n^1=t, \quad t<0; \qquad z_j^{\alpha}=0 \qquad {\rm otherwise;}
\]
then
\[
|z|=|t|, \qquad {\rm and} \qquad \nu M(x,\beta z)\le f(x,z)=|t|^p.
\]
Moreover, we consider the matrix $\tilde{z}$ such that 
\[
\tilde{z}_n^1=- \beta t, \quad - \beta t>0; \qquad \tilde{z}_j^{\alpha}=0 \qquad {\rm otherwise;}
\]
then $\tilde{z}=- \beta z$ and

\[
|\tilde{z}|=\beta |t|,  \qquad f(x, \tilde{z})= \beta^p |t|^p+ a(x)\beta^q |t|^q \le L\bigg( M(x,- \beta z)+g(x)\bigg)= L\bigg( M(x,\beta z)+g(x)\bigg),
\]
where the last equality is due to the symmetry of $M$ with respect to the second variable. Therefore, we get

\[
\beta^p |t|^p+ a(x)\beta^q |t|^q \le L\bigg( \frac{|t|^p}{\nu}+g(x)\bigg)
\]
that is false if $t \rightarrow - \infty$ and $a(x)>0$.

\noindent
We cannot apply the result in \cite{Koch-Ruf-Schaffner} since $q$ can be any number larger than $p$, provided $\sigma$ is chosen to be $n(q-p)/p$: this shows that assumption (2.2) in \cite{Koch-Ruf-Schaffner} fails.

\noindent
Note that assumption (3) of Remark 4.3 in \cite{HH} fails. Indeed, let us recall such an assumption. Set 
$$
A^\alpha_i (x,z) = \frac{\partial f}{\partial z^\alpha_i}(x,z)
= p |z|^{p-2} z^\alpha_i + a(x) q ( \max \{ z^1_n ; 0 \})^{q-1} \delta^{\alpha 1} \delta_{i n},
$$ 
where we used Kronecker delta: $\delta_{i n}=1$ when $i = n$ and $\delta_{i n}=0$ otherwise; a similar definition holds for $\delta^{\alpha 1}$.
Aforementioned assumption (3) requires the existence of a constant $L \geq 1$ such that
\begin{equation}
\label{assumption_HH}
|z'| |A(x,z')| \leq L A(x,z) z
\end{equation}
for every $x,z,z'$, with $|z'|=|z|$. 
We are going to show that \eqref{assumption_HH} fails. Indeed, let us take $z=-z'$ with 
$$
(z')^1_n = t > 0
\quad
{\rm and }
\quad
(z')^\alpha_i = 0
\quad
{\rm otherwise. }
$$
Then, $|z'|=|z|=t$, 
$A^1_n (x,z') = p t^{p-1} + a(x) q t^{q-1}$, 
$A^1_n (x,z) = - p t^{p-1}$ and $A^\alpha_i (x,z') = 0 = A^\alpha_i (x,z)$ otherwise.
In this case, \eqref{assumption_HH} becomes
\begin{equation}
\label{assumption_HH_bis}
t (p t^{p-1} + a(x) q t^{q-1}) \leq L p t^p.
\end{equation}
This inequality is false when $t$ goes to $+\infty$ since $p<q$, provided $x$ has been chosen in such a way that $a(x)>0$.

\vskip0,5cm
\noindent
{\bf Example 2}
\vskip0,5cm

\noindent We know that if $a(x)$ is sequentially lower 
 semicontinuous in $\Omega$
and
$$
f(x,z)=f_1(z)+ a(x)f_2(z)
$$ 
with $f_2(z)\ge 0$, then, property (F4) holds true:

\vskip0,4cm 
\noindent $\forall x\in \Omega$, $\forall \varepsilon >0$, with $\overline{B(x,\varepsilon)}\subset \Omega$, there exists $y^*\in \overline{B(x,\varepsilon)} $ such that 

$$f(y^*,z))\le f(y,z), \qquad \forall y \in  \overline{B(x,\varepsilon)}\quad {\rm and} \quad \forall z .$$ 
We ask ourselves whether there are $f(x,z)$ without product structure $ a(x)f_2(z)$,  satisfying property {(F4)}. The answer is positive. Indeed, for $q>1$, let us consider
$$g(x_1,t)=\max\{(\max \{t,0\})^q - (\max \{x_1,0\})^q;\, 0\}$$
and, for $1<p<q$, 

$$f(x,z)=|z|^p+ g(x_1,z_n^1)=|z|^p+\max\{(\max \{z_n^1,0\})^q - (\max \{x_1,0\})^q;\, 0\}.$$

\noindent Now we show that $g$ is not of product type: $g(x_1,t)=a(x_1)h(t)$. Indeed, if we assume $g(x_1,t)=a(x_1)h(t)$ then
$$a(x_1)h(t)=g(x_1,t)=t^q-(x_1)^q>0,$$
for all $x_1>0$ and $t>x_1>0$. Therefore,  $a(x_1)\ne 0$ and $ h(t)\ne 0$. Hence,  $a(x_1)\ne 0$
for all $x_1>0$. Moreover, for all $t>0$ we consider $x_1=\frac{t}{2}$ and we get $0<x_1<t$, so  $h(t)\ne 0$. This means that 

$$ {\rm if} \quad t>0 \qquad \text{then} \qquad h(t) \ne 0. $$
\noindent
On the other hand, if $0<t<x_1$ then $a(x_1)h(t)=g(x_1,t)=0$. Since $a(x_1)\ne 0$  for all $x_1>0$ it follows $h(t)=0$, that is,
$$ {\rm if} \quad 0<t<x_1 \qquad \text{then}\qquad h(t)= 0.$$
Moreover, for all $t>0$ we consider $x_1=2t$ and we get $0<t<x_1$ so $h(t)=0$. This means that
$${\rm if} \quad t>0 \qquad \text{then} \qquad h(t) = 0. $$ 
We get a contradiction and therefore $g(x_1,t) \ne a(x_1)h(t)$.

\noindent
Now, we easily check that
 
$g(x_1,t)=0$  when $t \leq 0$ or $0 \leq t \leq x_1$,

$g(x_1,t)=t^q$ when  $x_1 \leq 0 \leq t$,  

$g(x_1,t)=t^q-x_1^q$  when $0 \leq x_1 \leq t$. 

\vskip0.3cm

\noindent 
 Therefore, for $t\le 0 $, $g(x_1,t)=0$ for every $x_1$. On the other hand, for $t>0$ we have 

\begin{equation*}
g(x_1,t)=\begin{cases}
t^q &\quad {\rm if}\quad x_1\le 0\\
t^q-x_1^q &\quad {\rm if}\quad 0\le x_1\le t \\
0 &\quad {\rm if}\quad t\le x_1\,.
\end{cases}
\end{equation*}


\begin{figure}[h]
\centering
\scalebox{0.7}{
\begin{tikzpicture}
\begin{axis}[
        samples = 100,
        domain = -3:3,
        xmax=4, xmin=-3,
        ymax=1.85, ymin=-0.5,
        axis x line = center,
        axis y line = center,
        axis on top=true,
        xlabel = {$ {\scriptstyle x_1} $},
        ticks = none
        ]
        \addplot[name path=A, line width=2pt,solid] {max(1-(max(x,0))^2,0)
        } [yshift=3pt] node[pos=.95,left] 
        { 
        };
\end{axis}
\end{tikzpicture}}
\caption{Plot of $x_1 \to g(x_1,t)$ when $t=1$ and $q=2$}
 \label{figure2}
\end{figure}
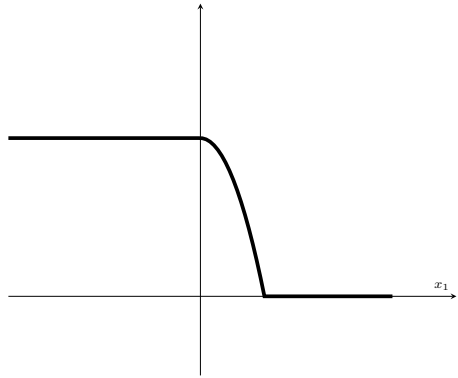

\noindent
By analyzing the function $v(x_1)=t^q-x_1^q$ for $x_1> 0$, we realize that it is decreasing and concave. Then $\forall y\in \overline{B(x,\varepsilon)}$, we get 
$$g(x_1+\varepsilon,t)\le g(y_1,t).$$

\noindent
Then, by choosing $y^*=(x_1+\varepsilon,x_2,\ldots,x_n)\in \overline{B(x,\varepsilon)}$ we can conclude that

$$f(y^*,z)=|z|^p+g(x_1+\varepsilon,z_n^1)\le |z|^p+ g(y_1,z_n^1)=f(y,z),$$
for every $y\in \overline{B(x,\varepsilon)}$ and (F4) holds true. 

\noindent
Let us recall that, given two convex functions $\psi_1,\,\psi_2:\R^k\to \R$, also the function  $\max \{\psi_1,\,\psi_2\}$ is convex.  Moreover, it is well known that if $v:\R\to [0, +\infty) $ is convex and  $w:[0, +\infty)\to [0, +\infty)$  is increasing and convex, then, composition $w\circ v:\R\to [0, +\infty) $ is convex.
Therefore we have that, for $q>1$,
the functions $$t\to (\max\{t,0\})^q,$$
$$t\to (\max\{t,0\})^q-(\max\{x_1,0\})^q$$ and 
$$t\to \max\{(\max\{t,0\})^q-(\max\{x_1,0\})^q; 0\}$$ are convex on $\R$, and we deduce that $f(x,z)$ is convex with respect to $z$: then (F3) holds true.

\noindent
\noindent
Furthermore,  since $g \ge0$,
$$|z|^p\le f(x,z);$$
then, the left hand side of (F2) holds true.

\noindent Now we prove that  

$$g(x_1,t)\le g(y_1,t)+R^q \qquad {\rm for \, every} \quad   x_1,y_1,t, \quad {\rm with }\quad |x_1|\le R,\; |y_1|\le R.$$

\noindent
Indeed, when $t\le 0$, we have

$$g(x_1, t)=0= g(y_1, t)\le g(y_1, t)+R^q; $$
if $t>0$ and $x_1\le 0$, we get
\begin{align*}
  g(x_1, t) &=t^q \\ &= t^q- 
(\max\{y_1;0\})^q+(\max\{y_1;0\})^q \\ & \le   \max\{t^q-(\max\{y_1,0\})^q;0\}+R^q \\ & = 
g(y_1, t)+R^q;
\end{align*}
if $0< x_1<t$, then, as before,

$$g(x_1, t)=t^q-x_1^q\le t^q\le g(y_1, t)+R^q\,; $$
if $0< t \leq x_1$, then, 

$$g(x_1, t)= 0 \le g(y_1, t)+R^q\,.$$
Therefore 
$$f(x,z)\le f(y,z)+R^q$$
is true $\forall z$ and $\forall x,y\in \overline{B(0,R)} \,.$
Then, also the right hand side of (F2) holds true with $K_1 = 1$, $K_2 = 0$, $K_3 = R^q$ and $\sigma$ can be any positive number.
If we restrict to $R=1$, then $K_3 =1$, so (F1) is satisfied by any pair $p,q$ with $1<p<q$, since we can take $\sigma = n(q-p)/p$ in (F1) and (F2).

\vskip0,5cm

\noindent
We cannot apply the result in \cite{Esposito-Leonetti-Petricca} since \eqref{Es-Leo-Petr sigma} fails as $q$ can be any number larger than $p$. For the same reason, assumption (2.2) in \cite{Koch-Ruf-Schaffner} fails.

\noindent We cannot apply the result in \cite{BC} to this functional, indeed the simmetry assumption with respect to the second variable is not satisfied as we prove choosing the following two matrices $z$ and $\tilde{z}$. First we consider the matrix $z$ such that

\[
z_n^1=1,  \qquad z_j^{\alpha}=0 \qquad {\rm otherwise}
\]
 and we take $x_1\le0$, then 
\[
|z|=1, \qquad g(x_1,z_n^1)= 1.
\]
Then we consider the matrix $\tilde{z}=-z$ such that

\[
(-z)_n^1=-1,  \qquad (-z)_j^{\alpha}=0 \qquad {\rm otherwise}
\]
and we take again $x_1\le0$, then 

\[
|-z|=|z|=1, \qquad g(x_1,(-z)_n^1)= 0.
\]
Therefore, if  $x_1\le0$ we have

$$f(x,z)=|z|^p+ g(x_1,z_n^1)=2 \qquad \text{and} \qquad f(x,-z)=|-z|^p+ g(x_1,(-z)_n^1)=1 . $$

\vspace{.2cm}
\noindent We cannot apply the result in \cite{BCM} to this functional. In other words, we prove that assumption \eqref{hp structure BCM} is not satisfied. For the convenience of the reader, we recall such an assumption:

\noindent  there exist positive constants $\nu,L >0$, and $\beta \in (0,1)$  such that
\begin{equation*}
 \nu M(x,\beta z)\le f(x,z) \le L(M(x,z)+1)
\end{equation*}
with $M(x,z)=M(x,-z)$. Indeed, if we consider the following matrix $z$ for which
\[
z_n^1=t \quad {\rm with } \quad t<0 \quad {\rm and } \qquad z_j^{\alpha}=0 \qquad {\rm otherwise,}
\]
then, $|z|=|t|$ and $g(x_1, z_n^1)=0$, so that
\[
\nu M(x,\beta z)\le f(x,z)=|z|^p+g(x_1, z_n^1)=|t|^p.
\]
If we consider the matrix $\tilde{z}=-\beta z$ we have
\[
\tilde{z}_n^1=- \beta t>0 \quad {\rm and } \qquad \tilde{z}_j^{\alpha}=0 \qquad {\rm otherwise,}
\]
then,
\[
|\tilde{z}|=\beta |t| \qquad {\rm and} \qquad f(x, \tilde{z})= f(x, - \beta z)=\beta^p |t|^p+g(x_1, - \beta t) = 
\beta^p |t|^p+(- \beta t)^q -(x_1)^q,
\]
provided $0<x_1<-\beta t$. If we take $x_1= \min \{\frac{1}{2}; -\beta \frac t 2 \}$ then $0<x_1 \le -\beta \frac t 2 < - \beta t$ and we are done. Moreover,
$$ f(x, - \beta z)\le L M(x,-\beta z)+L=L M(x,\beta z)+L \le L\frac{|t|^p}{\nu}+L$$
where we used also the symmetry of $M$ with respect to the second variable. Therefore we get

\[
\beta^p |t|^p+ \beta^q |t|^q -\biggm( \min \{\frac{1}{2}; -\beta \frac t 2 \}\biggm)^q = f(x,- \beta z)\le  L\bigg( \frac{|t|^p}{\nu}+1\bigg).
\]
We observe that $\min \{\frac{1}{2}; -\beta \frac t 2 \}=\frac{1}{2}$ as $t \to - \infty$, so from the last inequality we get
$$\beta^p + \beta^q |t|^{q-p} - \frac{1}{2^q |t|^p} \le  \frac{L}{\nu} +\frac{L}{|t|^p}$$
that is false if $|t| \rightarrow +\infty$ since $q>p$.

\noindent Finally, we prove that we cannot use the result in \cite{BCDFM}.
Indeed the structure assumption \eqref{hp_structure_BCDFM_new} cannot apply. Indeed, arguing by contradiction, let us first consider the matrix $z$ such that
\[
z_1^1=t>0 \quad {\rm and } \qquad z_j^{\alpha}=0 \qquad {\rm otherwise,}
\]
then,
\[
f(x,z)=|z|^p=t^p \qquad {\rm and} \qquad \nu_1\bigg( t^{\tilde{p}}+ \tilde{a}(x)t^{\tilde{q}}\bigg) \le  t^p
\]
therefore, as $t \rightarrow +\infty$,
\begin{align}\label{a(x)bisbis}
{\rm if }\quad \tilde{a}(x)=0&  \qquad {\rm then}  \qquad \tilde{p} \le p\notag\\
{\rm if} \quad \tilde{a}(x)>0 & \qquad  {\rm then}  \qquad \tilde{q} \le p.
\end{align}
If we fix $x_1$ and consider the matrix $\tilde{z}$ such that
\[
\tilde{z}_n^1=t>|x_1| \quad {\rm and } \qquad \tilde{z}_j^{\alpha}=0 \qquad {\rm otherwise,}
\]
we get $g(x_1,z_n^1)=t^q-(\max \{x_1;0\})^q$; 
then, 
\[
|\tilde{z}|=t \qquad {\rm and} \qquad f(x, \tilde{z})= t^p+ t^q - \biggm(\max \{x_1;0\} \biggm)^q \le \nu_2\bigg( t^{\tilde{p}}+ \tilde{a}(x)t^{\tilde{q}}\bigg).
\]
Therefore, as $t \rightarrow +\infty$, if $\tilde{a}(x)=0$  then $ q \le \tilde{p}$ that is impossible since by \eqref{a(x)bisbis} we get $\tilde{p} \le p< q$; on the other hand, if $\tilde{a}(x)>0$, then $q \le \tilde{q}$ that is impossible since by \eqref{a(x)bisbis} we get $\tilde{q} \le p< q$.

\noindent
Let us also note that $\frac{\partial g}{\partial t}(x_1,x_1)$ does not exist when $x_1>0$, so, assumption 1.5 in \cite{HH} fails.

\section{Energy approximation and absence of Lavrentiev phenomenon}
\label{en_approx_sec}

\noindent
We are going to show that absence of Lavrentiev phenomenon implies energy approximation. More precisely, 
 let us consider the integral functional \eqref{functional}, where
 $\Omega$ is an open ball $B$, 
 $x \to f(x,z) \geq 0$ is measurable and
 \begin{equation}
     \label{strict_convexity}
     z \to f(x,z) \quad \text{ is strictly convex.}
 \end{equation}
 Moreover, there exist constants $\tilde{c} > 0$ and $p>1$ such that
 \begin{equation}
     \label{p_growth_from_below}
    \tilde{c} (|z|^p - 1) \leq f(x,z).
 \end{equation}
We assume that $u \in W^{1,p}(B;\mathbb{R}^N)$, with $\mathcal{F}(u;B) < +\infty$, is a minimizer, that is,
\begin{equation}
     \label{minimality}
    \mathcal{F}(u;B) = \inf_{v \in u + W^{1,p}_0(B;\mathbb{R}^N)} \mathcal{F}(v;B).
 \end{equation}
Let us take $Y$ to be a subset of  $W^{1,p}(B;\mathbb{R}^N)$. Energy approximation means that
\begin{equation}
    \label{energy_approximation_sec4}
    \exists \{u_k\} \subset Y: \quad u_k \to u \text{ weakly in } W^{1,p}(B;\mathbb{R}^N) \quad \text{ and } \quad \mathcal{F}(u_k;B) \to \mathcal{F}(u;B).
\end{equation}
On the other hand, absence of Lavrentiev phenomenon means that
\begin{equation}
    \label{no_Lavrentiev_sec4}
    \inf_{v \in u + W^{1,p}_0(B;\mathbb{R}^N)} \mathcal{F}(v;B) =
    \inf_{v \in (u + W^{1,p}_0(B;\mathbb{R}^N)) \cap Y} \mathcal{F}(v;B).
\end{equation}
For the minimizer $u$, absence of Lavrentiev phenomenon implies energy approximation, that is, we have the following 
\begin{thm}
\label{absence_of_Lavrentiev} 
Under assumptions \eqref{strict_convexity}, \eqref{p_growth_from_below} and \eqref{minimality}, we have
\begin{equation}
    \label{no_Lav_implies_en_approx}
    \eqref{no_Lavrentiev_sec4}
    \qquad
    \Longrightarrow
    \qquad\eqref{energy_approximation_sec4}
\end{equation}
\end{thm}

\begin{proof}

\noindent
We use minimality \eqref{minimality} and absence of Lavrentiev phenomenon \eqref{no_Lavrentiev_sec4} in order to get 
$$
\mathcal{F}(u;B) = \inf_{v \in (u + W^{1,p}_0(B;\mathbb{R}^N)) \cap Y} \mathcal{F}(v;B).
$$
Then, there exists a sequence  $\{v_k\} \subset (u + W^{1,p}_0(B;\mathbb{R}^N)) \cap Y$ such that 
\begin{equation}
\label{convergence_sec4}
\mathcal{F}(v_k;B) \to \mathcal{F}(u;B).
\end{equation}
Now we need to prove the convergence of $v_k$ to $u$. To this aim, we use  
convergence \eqref{convergence_sec4} and we get $k_1$ such that
$$
 \mathcal{F}(v_k;B) \leq  \mathcal{F}(u;B) + 1,
$$
for every $k \geq k_1$. On the other hand, 
$p$ coercivity \eqref{p_growth_from_below} implies
$$
\tilde{c} \int_B (|Dv_k|^p - 1) dx \leq \mathcal{F}(v_k;B).
$$
These two inequalities merge into
$$
\tilde{c} \int_B (|Dv_k|^p - 1) dx \leq \mathcal{F}(u;B) + 1.
$$
Such an inequality says that $\{v_k - u \}$ is bounded in $W^{1,p}_0(B;\mathbb{R}^N))$; since $p>1$, there exists a limit function $w_0 \in W^{1,p}_0(B;\mathbb{R}^N))$ and a subsequence 
$\{v_{h_k} - u \}$ such that
$$
v_{h_k} - u \to w_0 \text{ weakly in } W^{1,p}_0(B;\mathbb{R}^N),
$$
\begin{equation}
\label{weak_conv}
v_{h_k}  \to u + w_0 \text{ weakly in } W^{1,p}(B;\mathbb{R}^N),
\end{equation}
\begin{equation}
\label{weak_conv_of_gradients}
Dv_{h_k}  \to D(u + w_0) \text{ weakly in } L^{p}(B;\mathbb{R}^N).
\end{equation}
Convexity \eqref{strict_convexity} implies lower semicontinuity with respect to weak convergence of gradients, so, \eqref{weak_conv_of_gradients} and \eqref{convergence_sec4} give
$$
\mathcal{F}(u + w_0;B) \leq \liminf \mathcal{F}(v_{h_k};B) 
= \lim \mathcal{F}(v_{k};B) = \mathcal{F}(u;B).
$$
Since $u$ is a minimizer and $u + w_0 \in u + W^{1,p}_0(B;\mathbb{R}^N)$, then $u + w_0$ is a minimizer too. Strict convexity \eqref{strict_convexity} implies uniqueness of minimizer, so, $w_0 = 0$ and weak convergence \eqref{weak_conv} tells us that
$$
v_{h_k}  \to u  \text{ weakly in } W^{1,p}(B;\mathbb{R}^N).
$$
We take $u_k = v_{h_k}$: this ends the proof.
\end{proof}

\noindent
\textbf{Remark 2.}
The last argument in the previous proof requires that all the $v_{h_k}$ have the same boundary value $u$.
When proving regularity of minimizers, energy approximation is an important tool but we do not need the same boundary value for all the approximating functions $u_k$. This makes things easier, as we can see in the proof of Theorem \ref{main result}.

\vskip0.5cm

\noindent
\textbf{Remark 3.}
Let us note that energy approximation can also be obtained when the density $f(x,z)$ can be approximated from below by $f_k(x,z)$ that increases with respect to $k$. In this case, the approximating functions $u_k$ are obtained by minimization of $v \to \int_B f_k(x,Dv(x)) dx$. Then, regularity of $u_k$ is linked to general regularity of minimizers: this requires additional assumptions, see \cite{DeFilippis-Leonetti-Treu}, \cite{Cupini-Guidorzi-Mascolo}, \cite{Cupini-Giannetti-Giova-Passarelli}, \cite{Carozza-Kristensen-Passarelli2014}.

\vskip0.5cm

\noindent
\textbf{Remark 4.}
Let us also mention that, in \cite{BuGwSk}, for the scalar case $N=1$, the authors first approximate $u$ by means of bounded $u_k$, then they use mollification for $u_k$. In the vectorial case $N \geq 2$, approximation by bounded functions is not easy, unless we are dealing with Uhlenbeck structure $f(x,z) = \tilde{f}(x,|z|)$, see section 7 in \cite{BuGwSk}; see also \cite{LeoPet}.

\vskip1cm

\noindent
\textbf{MSC 2020.} 49N99, 35J99, 35B65 \\ 
\textbf{Keywords.} Energy, approximation, calculus of variations, double phase \\
\textbf{Acknowledgements.} The authors are members of the Gruppo Nazionale per l’Analisi Matematica, la Probabilità e le loro Applicazioni (GNAMPA) of the Istituto Nazionale di Alta Matematica (INdAM). The third and fourth authors have been partially supported through the INdAM$-$GNAMPA 2024 Project “Interazione ottimale tra la regolarità dei coefficienti e l’anisotropia del problema in funzionali integrali a crescite non standard” (CUP: E53C23001670001).


\begin{thebibliography}{99}


\bibitem{Balci-Diening-Surnachev}
A. K. Balci, L. Diening, M. Surnachev.  
{\em New examples on Lavrentiev gap using fractals}
{Calc. Var. Partial Differential Equations}
 {\bf 59} (2020), no. 5, Paper No. 180, 34 pp.


\bibitem{Baroni-Colombo-Mingione} 
P. Baroni, M. Colombo, G. Mingione {\em Harnack inequalities for double phase functionals} 
Nonlinear Anal. {\bf 121} (2015), 206-222.

\bibitem{BBCLM} M. Borowski, P. Bousquet, I. Chlebicka, B. Lledos, B. Miasojedow, 
{\em Discarding Lavrentiev's Gap
 in Non-automonous and Non-Convex Variational Problems}
arXiv:2410.14995

\bibitem{BC} M. Borowski, I. Chlebicka, {\em Modular density of smooth functions in inhomogeneous and fully anisotropic Musielak–Orlicz–Sobolev spaces} J. Funct. Anal. {\bf 283}(12) (2022) 109716. https://doi.org/10.1016/j.jfa.2022.109716.

\bibitem{BCDFM} M. Borowski, I. Chlebicka, F. De Filippis, B. Miasojedow {\em Absence and presence of Lavrentiev’s phenomenon for
double phase functionals upon every choice of exponents}  Calc. Var. {\bf 63}, 35 (2024).  https://doi.org/10.1007/s00526-023-02640-1

\bibitem{BCM} M. Borowski, I. Chlebicka, B. Miasojedow, {\em Absence of Lavrentiev’s gap for anisotropic functionals}, Nonlinear Anal.,
{\bf 246}, (2024) 113584. https://doi.org/10.1016/j.na.2024.113584.

\bibitem{BMT2024}
P. Bousquet, C. Mariconda, G. Treu, 
{\em Non occurrence of the Lavrentiev
 gap for a class of nonautonomous functionals} Ann. Mat. Pura Appl. (4),
 {\bf 203} (5) (2024) 2275-2317.


\bibitem{BuGwSk}
M. Bulicek, P. Gwiazda, J. Skrzeczkowski. {\em On a Range of Exponents for Absence of Lavrentiev Phenomenon for Double Phase Functionals} {Arch. Rational Mech. Anal.} {\bf 246} (2022) 209-240.



\bibitem{Carozza-Kristensen-Passarelli2014}
M. Carozza, J. Kristensen, A. Passarelli di Napoli. 
{\em Regularity of minimizers of autonomous convex variational integrals}
{Ann. Sc. Norm. Super. Pisa Cl. Sci.} (5) {\bf 13} (2014), no.4, 1065-1089.

\bibitem{Colombo-Mingione} M. Colombo, G. Mingione 
{\em Regularity for double phase variational problems} Arch.
Ration. Mech. Anal. {\bf 215} (2015), no. 2, 443-496.

\bibitem{Colombo-Mingione-2} 
M. Colombo, G. Mingione {\em Bounded minimisers of double phase variational integrals} 
Arch. Ration. Mech. Anal. {\bf 218} (2015), no. 1, 219-273.

\bibitem{Cupini-Giannetti-Giova-Passarelli} 
G. Cupini, F. Giannetti, R. Giova, A. Passarelli di Napoli {\em Regularity results for
vectorial minimizers of a class of degenerate convex integrals} J. Differential Equations
{\bf 265} (2018), 4375-4416.

\bibitem{Cupini-Guidorzi-Mascolo} G. Cupini, M. Guidorzi, E. Mascolo {\em Regularity of minimizers of vectorial integrals
with p - q growth} Nonlinear Anal. {\bf 54} (2003), 591-616

\bibitem{DDP}
C. De Filippis, F. De Filippis, M. Piccinini {\em Bounded minimizers of double phase problems at nearly linear growth} Preprint (2024) \href{https://doi.org/10.48550/arXiv.2411.14325}{arXiv:2411.14325}

\bibitem{DeFilippis-Mingione}
C. De Filippis, G. Mingione {\em On the regularity of minima of non-autonomous functionals} 
J. Geom. Anal. {\bf 30} (2020), no. 2, 1584-1626

\bibitem{DeFilippis-Mingione-Arch2023}
C. De Filippis, G. Mingione
{\em 
Regularity for double phase problems at nearly linear growth}
Arch. Ration. Mech. Anal. {\bf 247}  (2023), no. 5, Paper No. 85, 50 pp.




\bibitem{DeFilippis-Piccinini}
F. De Filippis, M. Piccinini
{\em 
Regularity for multi-phase problems at nearly linear growth} 
J. Differential Equations {\bf 410} (2024), 832-868.

\bibitem{DeFilippis-Leonetti-Treu} 
F. De Filippis, F. Leonetti, G. Treu
{\em Nonoccurrence of Lavrentiev gap for a class
of functionals with nonstandard growth}
Adv. Nonlinear Anal. {\bf 13} (2024), no. 1, Paper No. 20240002, 17 pp.


\bibitem{Eleuteri-Marcellini-Mascolo}
M. Eleuteri, P. Marcellini, E. Mascolo 
{\em Regularity for scalar integrals without structure conditions} 
Adv. Calc. Var. {\bf 13} (2020), no.3, 279–300.

\bibitem{EspLeoMin-2004}
L. Esposito, F. Leonetti, G. Mingione. {\em Sharp regularity for functionals with (p, q)-growth} {J. Differential Equations} {\bf 204}
(2004), no. 1, 5-55.

\bibitem{Esposito-Leonetti-Petricca}
A. Esposito,  F. Leonetti,  P. V. Petricca. {\em Absence of Lavrentiev gap for non-autonomous functionals with (p,q)-growth} {Adv. Nonlinear Anal.} {\bf 8} (2019), 73-78.

 
\bibitem{FonMalMin}
I. Fonseca, J. Maly, G. Mingione. {\em Scalar minimizers with fractal singular sets} {Arch. Ration. Mech. Anal.} {\bf 172} (2004),
295-307.


\bibitem{HH} P. H\"ast\"o, J. Ok, {\em Regularity theory for non-autonomous partial differential equations without Uhlenbeck structure}, Arch. Rational Mech. Anal. , {\bf 245} (2022) 1401–1436. 

\bibitem{Kinnunen-Nastasi-Camacho}
J. Kinnunen, A. Nastasi, C. Pacchiano Camacho
{\em Gradient higher integrability for double phase problems on metric measure spaces} 
Proc. Amer. Math. Soc.{\bf 152} (2024), no.3, 1233–1251.

\bibitem{Koch}
L. Koch. {\em Global higher integrability for minimisers of convex functionals with (p,q)-growth} 
{Calc. Var. Partial Differential Equations} {\bf 60} (2021), no. 2, Paper No. 63.

\bibitem{Koch-Ruf-Schaffner}
L. Koch, M. Ruf, M. Schaffner, {\em On the Lavrentiev gap for convex, vectorial integral functionals} J. Funct. Anal. (2024),
110793, doi: https://doi.org/10.1016/j.jfa.2024.110793.

\bibitem{LeoPet}
F. Leonetti, P. V. Petricca, 
{\em Bounds for some minimizing sequences of functionals} 
Adv. Calc. Var. {\bf 4} (2011), no. 1, 83-100.

\bibitem{Marcellini1989}
P. Marcellini
{\em Regularity of minimizers of integrals of the calculus of variations with nonstandard growth conditions} {Arch. Rational Mech. Anal.} {\bf 105} (1989), no. 3, 267-284.


\bibitem{Mingione_survey2006}
G. Mingione
{\em Regularity of minima: an
invitation to the dark side of the calculus of variations} 
{Appl. Math.} {\bf 51}{ (2006) 355-426}.



\bibitem{TangQi}
Q. Tang. {\em Regularity of minimizers of nonisotropic integrals in the calculus of variations} {Ann. Mat. Pura Appl.} (4) {\bf 164}
(1993), 77-87.

\bibitem{Zhikov} V. V. Zhikov. {\em On Lavrentiev's phenomenon} {Russ. J. Math. Phys.} {\bf 3} (1995), 249-269.

\end{thebibliography}
\end{document}